\documentclass[bezier,amstex, oneside,reqno]{amsart}
\usepackage{amsmath, amssymb, amsthm, verbatim, euscript, amscd, textcomp, comment, upgreek}
\usepackage[dvips]{graphicx}
 \usepackage[all,cmtip]{xy}
\usepackage{epsfig}
\usepackage{color}

\def\ie{\emph{i.e., }}
\def\eg{\emph{e.g., }}

\def\S{\mathbb S}
\def\DD{\mathbb D}
\def\R{\mathbb R}
\def\Z{\mathbb Z}
\def\cZ{\mathcal Z}

\numberwithin{equation}{section}

\newtheorem*{problem}{Open Problem}
\newtheorem*{question}{Question}
\newtheorem*{main}{Main Theorem}
\newtheorem{theorem}{Theorem}[section]
\newtheorem{lemma}[theorem]{Lemma}
\newtheorem{prop}[theorem]{Proposition}

\newtheorem{cor}[theorem]{Corollary}
 \theoremstyle{remark}
\newtheorem{remark}[theorem]{Remark}

\theoremstyle{remark}

\renewcommand\labelenumi{\theenumi.}

\makeatletter
\renewcommand*\env@matrix[1][*\c@MaxMatrixCols c]{%
  \hskip -\arraycolsep
  \let\@ifnextchar\new@ifnextchar
  \array{#1}}
\makeatother

\begin{document}
\author{Andrey Gogolev}
\address{Binghamton University, Binghamton, NY 13902.}
\email{agogolev@math.binghamton.edu}
\author{Jean-Fran\c{c}ois Lafont$^\ast$}
\address{Ohio State University, Columbus, OH 43210.}
\email{jlafont@math.ohio-state.edu}
\title[Absence of Anosov diffeomorphisms]{Aspherical products which do not support Anosov diffeomorphisms}
\thanks{$^\ast$A.G. was partially supported by NSF grant DMS-1266282. J.-F.L. was partially supported by NSF grant
DMS-1510640. \\
MSC Primary 37D20; Secondary 55R10, 57R19, 37C25.}
\begin{abstract}
We show that the product of infranilmanifolds with certain aspherical closed manifolds do not support Anosov diffeomorphisms. As a special case, we obtain that products of a nilmanifold and negatively curved manifolds of dimension at least 3 do not support Anosov diffeomorphisms.
\end{abstract}
\date{}
 \maketitle

\section{Introduction}
Let $M$ be a smooth closed $n$-dimensional Riemannian manifold.
Recall that a diffeomorphism $f$ is called {\it Anosov} if there
exist constants $\lambda \in (0,1)$ and $C>0$ along with a
$df$-invariant splitting $TM=E^s\oplus E^u$ of the tangent bundle of
$M$, such that for all $m \ge 0$
\begin{multline}
\label{def_anosov}
\qquad\|df^mv\|\le C\lambda^m\|v\|,\;v\in E^s,\; \\
\qquad\shoveleft{\|df^{-m}v\|\le C\lambda^{m}\|v\|,\;v\in E^u.
\hfill}
\end{multline}

The invariant distributions $E^s$ and $E^u$ are called the {\it stable} and {\it unstable } distributions.
If either fiber of $E^s$ or $E^u$ has dimension $k$ with $k\le\lfloor n/2\rfloor$ then $f$ is
called a {\it codimension $k$} Anosov diffeomorphism. An Anosov diffeomorphism is called {\it transitive} if there exist a point whose orbit is dense in $M$.

All currently known examples of Anosov diffeomorphisms are conjugate to affine automorphisms of infranilmanifolds. { It is a famous classification problem, which dates back to Anosov and Smale, to decide whether there are other manifolds which carry Anosov diffeomorphisms. In particular, Smale~\cite{Sm} asked whether manifolds which support Anosov diffeomorphisms must be covered by Euclidean spaces. }
It is a more restrictive but still very interesting problem to classify Anosov diffeomorphisms on such manifolds. The goal 
of our current work is to evince certain coarse geometric obstructions (presence of negative curvature) to the existence 
of Anosov diffeomorphisms on manifolds which are covered by Euclidean spaces. 



\begin{main}\label{thm_nil} 
Let $N$ be a closed infranilmanifold and let $M$ be a smooth aspherical manifold whose fundamental group 
 $\Gamma=\pi_1(M)$ has the following three properties: (i) $\Gamma$ is Hopfian, (ii) $Out(\Gamma)$ is finite, 
 and (iii) the intersection of all maximal nilpotent subgroups of $\Gamma$ is trivial. Then $M\times N$ does not 
 support Anosov diffeomorphisms.
\end{main}
We will give a brief overview of known results in Section \ref{sec:previous-results}, and discuss background
material in Sections \ref{Sec_Lef}, \ref{sec_yano}, and \ref{sec_hyp_sets}. 
The proof of the {\bf Main Theorem} will be given in 
Section \ref{sec:proof-of-theorem}. Finally, in Section \ref{sec:conclusion}, we discuss further open problems. We
also provide some concrete classes of manifolds satisfying our {\bf Main Theorem}, for instance:

\begin{cor}\label{main-corollary}
Let $N$ be any closed infranilmanifold, and 
let $M_1, \ldots ,M_k$ be a collection of closed smooth aspherical manifolds  of dimension $\geq 3$, {
each of which satisfies one of the following properties:
\begin{enumerate}
\item it has Gromov hyperbolic fundamental group; or 
\item it is an irreducible higher rank locally symmetric space with no local $\mathbb H^2$-factors or $\mathbb R$-factors.
\end{enumerate}}
\noindent Then the product $M_1\times \cdots M_k \times N$ does not support any Anosov diffeomorphisms.
\end{cor}

Since closed negatively curved Riemannian manifolds are aspherical and have Gromov hyperbolic fundamental 
group, the corollary shows that any product of such manifolds of dimension $\ge 3$ with a nilmanifold does not 
support Anosov diffeomorphisms.

\vskip 10pt

\noindent  {\bf Acknowledgements.} We would like to acknowledge extremely helpful communications with 
Tom Farrell, Sheldon Newhouse and Federico Rodriguez Hertz. We also would like to thank Rafael Potrie and the anonymous referee for useful communications.

\section{Previous results}\label{sec:previous-results}

In this section, we provide the reader with a brief overview of the current state of the classification problem.

A full answer to the classification question was achieved under certain additional assumptions:

\begin{itemize}
\item Franks and Newhouse~\cite{Fr, N} proved that codimension one Anosov diffeomorphisms can only exist on manifolds which are homeomorphic to tori.
\item Franks and Manning~\cite{Fr2, Mann2} proved that Anosov diffeomorphisms on infranilmanifolds are conjugate to affine Anosov diffeomorphisms. In particular, Anosov diffeomorphisms on nilmanifolds are conjugate to hyperbolic automorphisms of nilmanifolds.
\item Brin and Manning~\cite{Br, BrM} showed that ``sufficiently pinched" Anosov diffeomorphisms can only exist on infranilmanifolds.
 \end{itemize}
 
In light of these results, one is interested in the classification of Anosov automorphisms of nilmanifolds. In general, such a classification seems to be hopeless. However it was achieved in low dimension (see~\cite{LW}, corrections in~\cite{D} and references therein).
 
A number of classes of manifolds are known to {\it not} support any Anosov diffeomorphisms, or to not support 
Anosov diffeomorphisms of a certain type:
 
 \begin{itemize}
\item Hirsch ~\cite{H} showed that certain manifolds with polycyclic fundamental group do not admit Anosov diffeomorphisms. In particular, he showed that mapping tori of hyperbolic toral automorphisms do not carry Anosov diffeomorphisms. 
\item Shiraiwa~\cite{Sh} noted that an Anosov diffeomorphism with orientable stable (or unstable) distribution cannot induce the identity map on homology in all dimensions. It follows, for example, that spheres, lens spaces and projective spaces do not admit Anosov diffeomorphisms.
\item Ruelle and Sullivan~\cite{RS} showed that if $M$ admits a codimension $k$ transitive Anosov diffeomorphism 
$f$ with orientable invariant distributions then $H^k(M; \mathbb R)\neq0$. In fact, they show that there is a {non-zero} 
cohomology class $\alpha\in H^k(M; \mathbb R)$ and positive $\lambda \neq 1$ with the property that 
$f^*(\alpha)= \lambda \cdot \alpha$. Existence of such a class provides an obstruction to the existence of Anosov diffeomorphisms on the level of the cohomology ring.
\item Gogolev and Rodriguez Hertz~\cite{GRH} have recently shown that certain products of spheres cannot 
support Anosov diffeomorphisms.
\item Finally, the most relevant reference with respect to the current paper is~\cite{Y}, where Yano used bounded cohomology to show that negatively curved manifolds do not support Anosov diffeomorphisms, (cf. Section~\ref{sec_yano} for an alternate proof).
\end{itemize}

\section{Periodic points of Anosov diffeomorphisms} \label{Sec_Lef}

Here we collect some well known facts. We refer the reader to~\cite[Chapter 6]{vick} and~\cite{Sm} for further details on this material.

Recall that if $X$ is a closed manifold and  $f\colon X\to X$ is a self-map with finitely many fixed points, then
the Lefschetz formula calculates the sum of indices of the fixed points --- the Lefschetz number --- as follows
$$
\Lambda(f)\stackrel{\mathrm{def}}{=}\sum_{p\in Fix(f)}ind_f(p)=\sum_{k\ge 0}(-1)^k Tr(f_*|_{H_k(X; \R)}).
$$
Now assume that $X$ is a closed oriented manifold and $f$ is an Anosov diffeomorphism with oriented unstable subbundle $E^u$, 
and that $f$ preserves the orientation of the unstable subbundle. Then 
$$ind_{f^m} (x)=(-1)^{\dim E^u}
$$
for all $x\in Fix(f^m)$, $m\ge 1$. Hence the number of points fixed by $f^m$ can be calculated as follows:
\begin{equation}
\label{eq_lefschetz}
\left|Fix(f^m)\right|=\left|\Lambda(f^m)\right|
\end{equation}


On the other hand,  $|Fix(f^m)|$ can be calculated from the Markov coding. In particular, for a transitive 
Anosov diffeomorphism $f$ the following asymptotic formula holds
\begin{equation*}\label{form_asymp_transitive}
|Fix(f^m)|=e^{mh_{top}(f)}+o(e^{mh_{top}(f)}),
\end{equation*}
where $h_{top}(f)$ is the topological entropy of $f$. 
 
For general, not necessarily transitive, Anosov diffeomorphism $f$ the formula takes the form 
\begin{equation}\label{form_asymp_general}
|Fix(f^m)|=re^{mh_{top}(f)}+o(e^{mh_{top}(f)})
\end{equation}
where $r$ is the number of transitive basic sets with entropy equal to $h_{top}(f)$.

Finally, we will need a formula due to Manning~\cite{Mann1} for the Lefschetz number of an automorphism of a
nilmanifold in terms of its eigenvalues. Let $L\colon N\to N$ be an automorphism of a compact nilmanifold $N$. Then
\begin{equation}\label{eq_manning}
\Lambda(L^m)=\prod_{\lambda\in \text{Spec}(L)} (1-\lambda^m), \;\;\;m\ge 1,
\end{equation}
where the product is taken over all eigenvalues (counted with multiplicity) of the Lie algebra automorphism induced 
by $L$.


\section{Absence of Anosov diffeomorphisms on negatively curved manifolds}
\label{sec_yano}

The following Lemma is presumably well-known to experts, but does not seem to appear in the literature.\footnote{We thank the referee for observing that the argument based on the Lefschetz formula works for proving Lemma~\ref{lem:finite-out}, cf. \tt{arXiv:1511.00261v1}}

\begin{lemma}\label{lem:finite-out}
Let $M$ be a closed aspherical manifold with fundamental group $\Gamma:= \pi_1(M)$. If $\text{Out}(\Gamma)$ is 
torsion, then $M$
does not support any Anosov diffeomorphism. 
\end{lemma}

\begin{proof}
Let us assume $M$ supports an Anosov diffeomorphism. Lifting to a finite cover, we may assume that the 
invariant distributions are oriented. For $M$ aspherical, there is an identification $H^*(M; \mathbb R) \cong 
H^*(\Gamma; \mathbb R)$ between the cohomology of the manifold $M$ and the (group) cohomology of $\Gamma$. 
Moreover, this identification is functorial, so one has a commutative diagram
$$
\xymatrix{
H^*(M; \mathbb R) \ar[r]^{f^*} \ar[d]^{\simeq}& H^*(M; \mathbb R) \ar[d]^{\simeq} \\
H^*(\Gamma ; \mathbb R) \ar[r]^{(f_\sharp)^*} & H^*(\Gamma ; \mathbb R) 
}
$$
At the level of group cohomology, any inner automorphism induces the identity map on group cohomology. 
Since $\text{Out}(\Gamma)$ is torsion, a finite power of $f_\sharp$ is inner, and hence a finite power $f_\sharp^k$ is the identity on cohomology. 

Therefore the cohomology automorphisms $(f^{mk})^*$, $m\ge 1$ are the identity on $H^*(M)$. Hence, using 
Poincar\'e Duality, we conclude that the Lefschetz numbers $\Lambda(f^{mk})$, $m\ge 1$ are uniformly bounded. 
But this is in contradiction with the growth of periodic points, by~(\ref{eq_lefschetz}) and~(\ref{form_asymp_general}).
\end{proof}
\begin{remark} There exist aspherical manifolds $M$ which do not admit Anosov diffeomorphisms for the same reason, but have infinite order elements in $\text{Out}(\pi_1M)$. Indeed, the manifolds $M$ obtained from a certain (twisted) double construction as in~\cite{NP}, 
\cite[Sections 5.2, 5.4]{FLS}, have infinite order elements in $\text{Out}(\pi_1M)$. However, it is easy to see that (after passing to a finite power) all infinite order elements act trivially on cohomology. Hence, by the above argument, such $M$ also do not admit Anosov diffeomorphisms.
\end{remark}

We note that Lemma \ref{lem:finite-out} already gives a substantial restriction on the possible closed aspherical 
manifolds that can support an Anosov diffeomorphism. For instance, it follows from Mostow rigidity that the outer 
automorphism group of a lattice $\Gamma$ in a simply-connected, connected, non-compact, real semisimple Lie group $G$ 
has finite outer automorphism provided $G$ has no local $SO(2,1)$ factors. This yields the immediate:

\begin{cor}
Let $M$ be a non-positively curved locally symmetric manifold, whose universal cover $\tilde M$ does not split 
off an $\mathbb H^2$-factor or an $\R$-factor. Then $M$  cannot support any  Anosov diffeomorphisms.
\end{cor}

\begin{remark} 
 There are hyperbolic 3-manifolds $M$ which are known to support Anosov flows (see Goodman~\cite{Goo}). Thus, 
 the analogous result does not hold for Anosov flows or for partially hyperbolic diffeomorphisms. Note also, that by taking the product of the time-1 map of such a flow and an Anosov diffeomorphism of a nilmanifold $N$ we can obtain a partially hyperbolic diffeomorphism on the product $M\times N$ with one dimensional center distribution (cf. the Main Theorem).
\end{remark}

As another application of this result we recover the result of Yano~\cite{Y}.

\begin{cor}\label{cor-yano}
Let $M$ be a closed negatively curved Riemannian manifold of dimension $n\geq 3$. Then $M$ does not support any  Anosov diffeomorphism.
\end{cor}

\begin{proof}
Since $M$ is negatively curved, $\Gamma:= \pi_1(M)$ is a torsion-free Gromov hyperbolic group. Moreover, the 
boundary at infinity $\partial ^\infty \Gamma$ is homeomorphic to the boundary at infinity $\partial ^\infty \tilde M$ 
of the universal cover of $M$, which we know is a sphere $S^{n-1}$ of dimension $n-1 \geq 2$. If $M$ supports 
an Anosov diffeomorphism, then  Lemma~\ref{lem:finite-out} tells us $Out(\Gamma)$ must be infinite. Work of Paulin \cite{Pa} 
and Bestvina and Feighn \cite[Corollary 1.3]{BF} then implies that $\Gamma$ splits over a cyclic subgroup. By Bowditch 
\cite[Theorem 6.2]{Bow}, this forces $\partial ^\infty \Gamma \cong S^{n-1}$ to have a local cut-point (i.e., an open 
connected subset $U\subset S^{n-1}$ and a point $p\in U$ with the property that $U\setminus \{p\}$ is disconnected), 
a contradiction since the boundary at infinity is a sphere of dimension $\geq 2$.
\end{proof}



\section{Locally maximal hyperbolic sets}\label{sec_hyp_sets}
Let $f\colon M\to M$ be a diffeomorphism. Recall that an $f$-invariant closed set $K$ is called {\it hyperbolic} if the tangent 
bundle over $K$ admits a $df$-invariant splitting $T_KM=E^s\oplus E^u$, where vectors in $E^s$ ($E^u$) decay 
(grow) exponentially fast -- in the sense of equation~(\ref{def_anosov}). A hyperbolic set $K$ is called {\it locally 
maximal} if there exists an open neighborhood $V$ of $K$ such that
\begin{equation}
\label{eq_loc_maximal}
K=\bigcap_{m\in\Z}f^m(V).
\end{equation}
We note that $f$ is Anosov if and only if the entire manifold $M$ is a hyperbolic set. Thus, the notion of hyperbolicity is 
intended to reflect some ``Anosov-type'' behavior on the set $K$.

\begin{prop}
\label{prop_local_prod_structure}
A hyperbolic set $K$ is locally maximal if and only if it has a local product structure.
\end{prop}
We refer to~\cite[Chapter 5]{BS} for the definition of local product structure and the proof of this proposition.

\begin{prop}[see \eg Corollary 6.4.19 in~\cite{KH} ]
\label{prop_periodic_dense}
Let $K$ be a locally maximal hyperbolic set for $f$ and let $\Omega\subset K$ be the set of non-wandering points 
for the restriction $f|_K$. Then periodic points of $f|_K$ are dense in $\Omega$.
\end{prop}
We will use the above proposition for the proof of the following result.
\begin{prop}
\label{prop_periodic}
Let $K$ be an uncountable locally maximal hyperbolic set for $f$. Then the restriction $f|_K$ has infinitely many 
periodic points.
\end{prop} 
\begin{proof}
Let $\Omega$ be the set of non-wandering points of $f|_K$. (Note that $\Omega$ may not be the same as the set of 
non-wandering points of $f$ in $K$.) By  Proposition~\ref{prop_periodic_dense} periodic points of $f|_K$ are dense in 
$\Omega$. Hence, if $\Omega$ is infinite, then Proposition~\ref{prop_periodic} follows. 
 
Now assume that $\Omega$ consists of finitely many periodic orbits. By passing to a finite iterate of $f$ we  can also 
assume that $f(p)=p$ for all $p\in\Omega$. Denote by $W^{s,K}(p)$ and $W^{u,K}(p)$ the stable and unstable sets of 
$p\in\Omega$, respectively; that is,
$$
W^{s, K}(p)=W^s(p)\cap K,\;\; W^{u, K}(p)=W^u(p)\cap K,
$$
where $W^s(p)$ and $W^u(p)$ are the stable and unstable manifolds of $p$, respectively. Then $K$ decomposes 
as a disjoint union
 \begin{equation*}
 K=\bigsqcup_{p\in\Omega}W^{s,K}(p).
 \end{equation*}
Indeed, if a point $x\in K$ does not belong to any stable set, then neither does its forward orbit $\{f^n(x);n\ge 1\}$. Let $q$ be an $\omega$-limit point for $\{f^n(x);n\ge 1\}$. But then $q$ is a non-wandering point for $f|_K$, which does not belong to $\Omega$, yielding a contradiction. Similarly,
 \begin{equation*}
 K=\bigsqcup_{p\in\Omega}W^{u,K}(p).
 \end{equation*}
Therefore, each wandering point of $K$ is a heteroclinic (or homoclinic) point. However, the stable and unstable manifolds 
of points in $\Omega$ intersect in at most countably many points. Hence $K$ is countable, which yields a contradiction.
\end{proof}


\section{Proof of the Main Theorem}\label{sec:proof-of-theorem}
Our proof proceeds by assuming that there exists an Anosov diffeomorphism $f\colon M\times N\to M\times N$.
We first will go through a series of reductions (R1)-(R4).

\subsection{A reduction}
Here we show that, by passing to finite iterates of $f$ and finite covers of $M\times N$ we can additionally 
assume the following.
\begin{enumerate}
 \item[R1.] {\it $N$ is a nilmanifold;}
 \item[R2.] {\it the stable and unstable subbundles $E^u$ and $E^s$ are oriented and $f$ preserves these orientations;}
\end{enumerate}
Let $N'$ be a closed nilmanifold cover of $f$. Then some iterate of $f$ lifts to $M\times N'$ because of the following 
more general assertion.
\begin{lemma}
 Let $X$ be a closed manifold, let $\tilde X\to X$ be a finite cover of $X$ and let $f\colon X\to X$ be a diffeomorphism. 
 Then there exists $n\ge 0$ such that $f^n\colon X\to X$ lifts to a diffeomorphism $\widetilde{f^n}\colon\tilde X\to\tilde X$; 
 \ie the diagram 
$$
\xymatrix{
\tilde X\ar[d]\ar^{\widetilde{f^n}}[r] & \tilde X\ar[d] \\
X\ar^{f^n}[r]  & X
}
$$
commutes.
\end{lemma}
\begin{proof}
By the lifting criterion it suffices to show that some iterate of $f$ preserves the conjugacy class of the covering subgroup. 
The covering subgroup is a finite index subgroup of the fundamental group of $X$. Because the fundamental group of 
$X$ is finitely generated, there are only finitely many subgroups of given finite index (by a theorem of M. Hall). 
Hence, by the pigeonhole principle, some iterate of $f$ fixes the conjugacy class of the covering subgroup.
\end{proof}

Now, if $E^u$ is not orientable we can pass to the orienting double covering and to $f^2$, which ensures 
that the unstable distribution
are orientable and $f$ preserves the orientation. If $E^s$ is still non-orientable then the same procedure can 
be applied once again. 

In fact, we might need to pass to a double cover (and to a finite iterate) yet once more in order to be able to keep the 
product structure of the total space. Indeed, let $\Gamma=\pi_1M$ and $G=\pi_1N$. Let $H\subset \Gamma\times G$ be 
the orienting double cover subgroup. Then $H$ has index two and we let 
$H_\Gamma=H\cap (\Gamma\times \{id_G\})$ and $H_G=H\cap(\{id_\Gamma\}\times G)$. We identify $H_\Gamma$ 
and $H_G$ with subgroups of $\Gamma$ and $G$, respectively. It is easy to see that $H_\Gamma$ and $H_G$ are 
either index one or index two subgroups. If one of these subgroups has index one, then $H=H_\Gamma\times H_G$ 
and the double cover has a product structure. Otherwise $H_\Gamma\times H_G$ is an index two subgroup of $H$. 
In this case we pass to the 4-fold cover that corresponds to $H_\Gamma\times H_G$, which is clearly a product. 
Because $H_\Gamma\times H_G\subset H$ this 4-fold cover is also orienting.

\subsection{Further reduction: the induced automorphism}

Again, by passing to a finite iterate if necessary, we can assume that $f$ has a fixed point $p$. Consider the 
induced automorphism
$$
f_\#\colon \pi_1(M\times N, p)\to\pi_1(M\times N, p)
$$
Let $\Gamma=\pi_1M$ and $G=\pi_1N$. We identify $\pi_1(M\times N, p)$ and $\Gamma\times G$ in the obvious way.
\begin{lemma}
\label{lemma_automorphism}
The induced automorphism $f_\#$ has the following form
$$
f_\#(\gamma, g)=(\alpha(\gamma), \rho(\gamma)L(g)),
$$
where $\alpha\colon\Gamma\to\Gamma$ and $L\colon G\to G$ are automorphisms and $\rho\colon \Gamma\to\cZ(G)$ 
is a homomorphism into the center of $G$.
\end{lemma}
\begin{proof}
{
Recall that, by hypothesis, the group $\Gamma$ has the property that the intersection of all its maximal nilpotent 
subgroups is trivial. This implies that, for the group $\Gamma\times G$, the intersection of all the maximal nilpotent
subgroups is precisely the subgroup $\{id_\Gamma\}\times G$. It follows that the subgroup $\{id_\Gamma\}\times G$ 
is a characteristic subgroup in the product, hence is invariant under $f_{\#}$. We denote by 
$L = f_\#|_{\{id_\Gamma\}\times G} \in \text{Aut}(G)$ the induced automorphism of $G$.

Next we define $\alpha$ and $\rho$ via the following formula: for $\gamma \in \Gamma$, we set 
$$f_\# \left((\gamma, id_G)\right) = \left(\alpha(\gamma), \rho(\gamma) \right),$$
where $\alpha: \Gamma \rightarrow \Gamma$, $\rho:\Gamma \rightarrow G$ are homomorphisms.
Let us verify the expression for $f_\#$ given in the Lemma.
If $(\gamma, g)\in \Gamma \times G$ is arbitrary, then we have:
\begin{align*}
f_\#((\gamma, g)) &= f_\# \big((\gamma, id_G) (id_\Gamma, g) \big) 
=  f_\# \big((\gamma, id_G)\big) f_\# \big((id_\Gamma, g) \big) \\
&=  \left(\alpha(\gamma), \rho(\gamma) \right)(id_\Gamma, L(g)) = (\alpha(\gamma), \rho(\gamma)L(g)).
\end{align*}
Now $\alpha$ is clearly surjective, and $\Gamma$ is Hopfian, so $\alpha \in \text{Aut}(\Gamma)$. Finally, since
$f_\#$ is a homomorphism, the fact that $(\gamma, id_G)$ and $(id_\Gamma, g)$ always commute tells us that
$\rho(\gamma) L(g) = L(g) \rho(\gamma)$. This forces the homomorphism $\rho$ to have image in the center
$\cZ(G)$, and concludes the proof of the Lemma.
}
\end{proof}

Because $\text{Out}(\Gamma)$ is finite, by passing to a further iterate of $f$, we can (and do) assume the following.
\begin{enumerate}
\item[R3.]
{\it $\alpha\colon\Gamma\to\Gamma$ is an inner automorphism.}
\end{enumerate}

\subsection{The model}
Recall that, by work of Mal$'$cev~\cite{Mal}, a closed nilmanifold $N$ can be identified with the quotient space 
$\tilde N/G$, where $\tilde N$ is a simply connected nilpotent Lie group and $G\subset \tilde N$ is a cocompact 
lattice (which we identify with $\pi_1N$). Also $\cZ(G)=G\cap\cZ(\tilde N)$.

Now let $L\colon G\to G$ be an automorphism, then, again by work of Mal$'$cev, $L$ uniquely extends to an 
automorphism of $\tilde N$, which we continue to denote by $L$.

Denote by $\tilde M$ the universal cover of $M$. The fundamental group $\Gamma=\pi_1M$ acts on $\tilde M$ 
cocompactly by deck transformations, and we identify $\Gamma$ with an orbit of a base-point in $\tilde M$.

\begin{lemma}
 \label{lemma_rho_extend}
 Let $\rho\colon\Gamma\to\cZ(G)$ be a homomorphism. Then it extends to a smooth equivariant map 
 $\rho\colon\tilde M\to\cZ(\tilde N)$,
\ie
 $$
 \forall \gamma\in\Gamma,\;\;\forall x\in \tilde M\;\; \rho(\gamma x)=\rho(\gamma)\rho(x).
 $$
\end{lemma}
\begin{proof}
 Triangulate $\tilde M$ in an equivariant way so that $\Gamma\subset \tilde M$ belongs to the 0-skeleton. 
 Extend $\rho$ to the 0-skeleton equivariantly in an arbitrary way. Recall that the center $\cZ(\tilde N)$ is a 
 Euclidean space. Because $\cZ(\tilde N)$ is connected $\rho$ can be equivariantly extended to 1-skeleton. 
 And because $\cZ(\tilde N)$ is aspherical we can extend $\rho$ equivariantly to all skeleta by induction on 
 dimension. To finish the proof we approximate the resulting map by a smooth equivariant map. It is easy to 
 see (by using charts) that this can be done without changing the values on $\Gamma\subset \tilde M$.
\end{proof}

Let $L\colon G\to G$ and $\rho \colon \Gamma\to\cZ(G)$ be given by Lemma~\ref{lemma_automorphism}. 
Then $\rho$ extends to $\rho\colon\tilde M\to\cZ(\tilde N)$ by Lemma~\ref{lemma_rho_extend}. It is straightforward 
to check that the map
$$
\tilde M\times \tilde N\ni (x,y)\mapsto (x, \rho(x)L(y))\in \tilde M\times \tilde N
$$
descends to a (model) map $\bar f\colon M\times N\to M\times N$. We continue abusing notation and denote by 
$L$ the induced automorphism $L\colon N\to N$ and by $\rho$ the induced map $\rho\colon M\to \cZ(\tilde N)/\cZ(G)$. 
Then we can write
$$
\bar f(x,y)=(x,\rho(x)L(y)).
$$
By construction, the induced homomorphism $\bar f_\#\colon \pi_1(M\times N)\to\pi_1(M\times N)$ is given by
$$
\bar f_\#(\gamma,g)=(\gamma,\rho(\gamma)L(g)).
$$
Hence, by the discussion in the previous subsection, $f_\#$ and $\bar f_\#$ are conjugate. Therefore $f$ and 
$\bar f$ are homotopic maps. 

\begin{remark} Note that because $\rho$ is smooth the model map $\bar f$ is, in fact, a diffeomorphism. However 
this is not used in the sequel. We only use smoothness of $\rho$ for Lemma~\ref{lemma_theta}.
\end{remark}

\subsection{Hyperbolicity of the model}
\label{sec_34}
Consider a gradient vector field on $M$ with finitely many fixed points $q_1, q_2,\ldots ,q_k$, each of which is 
hyperbolic. Denote by $i_1, i_2,\ldots ,i_k$ the dimensions of the unstable manifolds at $q_1, q_2,\ldots ,q_k$, 
respectively. Recall that the classical Poincar\'e-Hopf theorem yields the following formula for the Euler characteristic:
\begin{equation}
\label{eq_euler}
\chi(M)=\sum_{j=1}^k ind(q_j)=\sum_{j=1}^k(-1)^{i_j}
\end{equation}
Denote by $\varphi$ the time-one map of the gradient flow and define $\bar{\bar f}\colon M\times N\to M\times N$ 
as follows
$$
\bar{\bar f}(x,y)=(\varphi(x),\rho(x)L(y))
$$
By a direct calculation one can see that 
$$
(\bar{\bar f}\,)^m(x,y)=(\varphi^m(x), \rho_m(x)L^m(y)),
$$
where $\rho_m\colon M\to \cZ(\tilde N)/\cZ(G)$ has an explicit expression in terms of $\rho$ and $L$.
Note that, for each of the maps $(\bar{\bar f}\,)^m$, the only possible fixed points must lie on the fibers
$\{q_i\} \times N$. 

We now proceed to modify the maps $(\bar{\bar f}\,)^m$, in order to control the fixed points on each of the fibers.
For each $m\ge 1$ and each $j=1,\ldots ,k$ consider a small perturbation $\widetilde{L^m_j}$ of 
$\rho_m(q_j)L^m\colon N\to N$ such that $\widetilde{L^m_j}$ has finitely many fixed points, each of which is 
hyperbolic. Now consider a small perturbation of $(\bar{\bar f}\,)^m$ of the form
$$
\widetilde{f^m}(x,y)=(\varphi^m(x),\widetilde{L^m_x}(y))
$$
such that $\widetilde{L^m_{q_j}}=\widetilde{L^m_j}$ for $j=1,\ldots ,k$. Note that we have homotopies
$\widetilde{L^m_j}\simeq \rho_m(q_j)L^m\simeq L^m$ and hence
\begin{equation}
\label{eq_fiber_lefschetz}
\Lambda(\widetilde{L^m_j})=\Lambda(L^m)
\end{equation}
Also note that by construction $\widetilde{f^m}$ has finitely many fixed points each of which has the form $(q_j,y)$, 
$j=1,\ldots ,k$. Moreover,
\begin{equation}
\label{eq_index_formula}
ind_{\widetilde{f^m}}(q_j,y)=(-1)^{i_j}ind_{\widetilde{f^m}|_{\{q_j\}\times N}}(y)=(-1)^{i_j}ind_{\widetilde{L^m_j}}(y)
\end{equation}

From the homotopies $f\simeq \bar f \simeq \bar{\bar f}$ and $(\bar{\bar f}\,)^m\simeq \widetilde{f^m}$ we can carry 
out the following calculation of the Lefschetz number.
\begin{align*}
\Lambda(f^m)=\Lambda(\widetilde{f^m})&=\sum_{(q_j,y)\in Fix(\widetilde{f^m})}ind_{\widetilde{f^m}}(q_j,y)
\stackrel{(\ref{eq_index_formula})}{=}\sum_{j=1}^k\left[(-1)^{i_j}\sum_{y\in Fix(\widetilde{L^m_j})}ind_{\widetilde{L^m_j}}(y)\right]\\
&=\sum_{j=1}^k(-1)^{i_j}\Lambda({\widetilde{L^m_j}})\stackrel{(\ref{eq_fiber_lefschetz})}{=}\sum_{j=1}^k(-1)^{i_j}\Lambda(L^m)
\stackrel{(\ref{eq_euler})}{=}\chi(M)\Lambda(L^m)
\end{align*}
Note that if the Euler characteristic $\chi(M)$ vanishes (for instance, if $M$ is odd-dimensional) then we immediately obtain a contradiction with~(\ref{form_asymp_general}){: the Lefschetz number simultaneously must grow exponentially fast, and must equal zero. This already completes the proof of the {\bf Main Theorem} in the special case where $\chi(M)=0$. The rest of the paper deals with the case where $\chi(M)\neq 0$.

\vskip 5pt

Now if $\chi(M)\neq 0$}, by combining the above formula with~(\ref{eq_lefschetz}), (\ref{form_asymp_general}) and (\ref{eq_manning}) we obtain the following equality for all $m\ge 1$.
$$
\frac{r}{\chi(M)}e^{mh_{top}(f)}+o(e^{mh_{top}(f)})=\prod_{\lambda \in spec(L)}|1-\lambda^m|
$$
Modulo the coefficient $r/\chi(M)$ this equality is the same as equation~(2) in~\cite{Mann2}. Therefore, 
the argument of Manning~\cite[pp. 425-426]{Mann2} yields the following result.
\begin{lemma}
\label{lemma_hyperbolic}
The eigenvalues of $L$ lie off the unit circle, \ie $L\colon N\to N$ is an Anosov automorphism.
\end{lemma} 

\subsection{Further reduction: a global change of coordinates}
The restriction of the map
$$
\tilde N\ni y\mapsto L(y)y^{-1}\in \tilde N
$$
to the center of $\tilde N$ is a homomorphism. By Lemma~\ref{lemma_hyperbolic}, the restriction of $L$ to $\cZ(\tilde N)$ 
is hyperbolic, and hence, the above homomorphism is, in fact, an automorphism. {Note that, since $\cZ(\tilde N)$ 
is a Euclidean space, the restriction of $L$ to $\cZ(\tilde N)$ can be represented by a matrix (also denoted $L$). The 
map $y\mapsto L(y)y^{-1}$ on $\cZ(\tilde N)$ is then represented by the matrix $L-id$.}
Denote by $(L-id)^{-1}$ the inverse of the restriction of this automorphism to $\cZ(G)\subset\cZ(\tilde N)$. 
Define a homomorphism $\theta\colon\Gamma\to\cZ(G)$ by
$$
\theta(\gamma)=(L-id)^{-1}\rho(\gamma).
$$
And define an automorphism $h_\#\colon \Gamma\times G\to\Gamma\times G$ by
\begin{equation}
\label{eq_h_sharp}
h_\#(\gamma,g)=(\gamma, \theta(\gamma)g).
\end{equation}
It is straightforward to check that 
$$
f^{new}_\#\stackrel{\mathrm{def}}{=}h_\#\circ f_\#\circ  h_\#^{-1}
$$
is given by
$$
f^{new}_\#=(\alpha(\gamma),L(g)).
$$
\begin{lemma}\label{lemma_theta}
There exists a self diffeomorphism $h\colon M\times N\to M\times N$ such that the induced automorphism 
$h_\#\colon \Gamma\times G\to\Gamma\times G$ is given by~(\ref{eq_h_sharp}).
\end{lemma}
\begin{remark}
{-- see Farrell and Jones} \cite{FJ}.
\end{remark}
\begin{proof}
Recall that, by Lemma~\ref{lemma_rho_extend}, the homomorphism $\rho\colon \Gamma\to\cZ(G)$ 
extends to a smooth equivariant map $\rho\colon\tilde M\to \cZ(G)$. Hence, by letting
$$
\theta(x)=(L-id)^{-1}\rho(x)
$$
we extend $\theta\colon\Gamma\to\cZ(G)$ to a smooth equivariant map, which in turn descends to a map 
$\theta\colon M\to\cZ(G)$. Then the formula
$$
h(x,y)=(x,\theta(x)y)
$$
defines the posited diffeomorphism.
\end{proof}

Now let
$$
f^{new}=h\circ f\circ h^{-1}.
$$
Clearly $f^{new}$ is also an Anosov diffeomorphism. Thus, by replacing $f$ with $f^{new}$ if necessary, we can 
assume the following.
\begin{enumerate}
\item[R4.]
{\it 
The induced automorphism $f_\#\colon \pi_1(M\times N, p)\to\pi_1(M\times N, p)$ has the following form
$$
f_\#(\gamma, g)=(\alpha(\gamma), L(g)).
$$}
\end{enumerate}

\subsection{A locally maximal hyperbolic set $K$ by applying Franks' theorem}
\label{sec_franks}
Consider the diagram
$$
\xymatrix{
\Gamma\times G\ar[d]\ar^{f_\#}[r] & \Gamma\times G \ar[d] \\
G\ar^L[r] & G
}
$$
By the reduction in the previous subsection this diagram commutes. Also, recall that $L\colon N\to N$ is an 
Anosov automorphism, and hence, is a $\pi_1$-diffeomorphism in the sense of Franks. Then Franks' thesis~\cite{Fr} 
yields a semi-conjugacy $h\colon M\times N\to N$, which induces $h_\#\colon (\gamma,g)\mapsto g$ and makes the 
diagram 
$$
\xymatrix{
M\times N\ar_h[d]\ar^{f}[r] & M\times N \ar_h[d] \\
N\ar^L[r] & N
}
$$
commute.

Define
$$
K=h^{-1}(h(p)).
$$
(Recall that $p$ is a fixed point of $f$.) Clearly $K$ is an $f$-invariant closed set.

\subsection{$K$ contains infinitely many periodic points} 

\begin{lemma}
\label{lemma_loc_maximal}
The set $K$ is a locally maximal hyperbolic set. (Hence has local product structure by Proposition
\ref{prop_local_prod_structure}.)
\end{lemma}
\begin{proof}
The set $K$ is hyperbolic because $f$ is Anosov. We check that $K$ is locally maximal by using the 
definition~(\ref{eq_loc_maximal}).
The point $h(p)$ is a locally maximal hyperbolic set for $L$, \ie
$$
h(p)=\bigcap_{m\in\Z}L^m(U),
$$
where $U$ is a sufficiently small neighborhood of $h(p)$. Let $V=h^{-1}(U)$, then
\begin{align*}
K=h^{-1}(h(p))&=h^{-1}\left(\bigcap_{m\in\Z}L^m(U)\right)=\bigcap_{m\in\Z}h^{-1}(L^m(U))\\
&=\bigcap_{m\in\Z} f^k(h^{-1}(U))=\bigcap_{m\in\Z} f^k(V).\\
\end{align*}
establishing the Lemma.
\end{proof}
We will now exploit the fact that the dimension of $M$ is positive.
\begin{lemma}
\label{lemma_uncountable}
The set $K$ is uncountable.
\end{lemma}
\begin{proof}
For each $x\in M$ consider the restriction $h|_{\{x\}\times N}\colon N\to N$. This restriction induces an 
isomorphism of fundamental groups. {The isomorphism of fundamental groups induces an isomorphism
on top-degree (group) cohomology. Since $N$ is aspherical, this implies the map induces an isomorphism between
the top-degree cohomology of the spaces. Hence the map has degree one, so must be onto. We conclude that 
$h^{-1}(h(p))\cap (\{x\}\times N)\neq\varnothing$,  for each $x\in M$. Since $\dim(M)>0$, this shows $K$ is uncountable.}
\end{proof}

We conclude that Proposition~\ref{prop_periodic} applies to $K$ and yields the following lemma.
\begin{lemma}\label{lemma_periodic_points}
There are infinitely many periodic points in the invariant set $K$.
\end{lemma}

\subsection{Lifting the dynamics to $M\times\tilde N$}
Let $\bar\pi\colon\tilde N\to N$ be the universal covering. By taking the product with the identity map $id_M$ we 
obtain the covering
$$
\pi\colon M\times\tilde N\to M\times N.
$$ 
Choose a base point $\tilde p\in M\times\tilde N$ such that $\pi(\tilde p)=p$. Then, using reduction R4, we see that 
$f\colon M\times N\to M\times N$ uniquely lifts to $\tilde f\colon M\times \tilde N\to M\times\tilde N$ with 
$\tilde f(\tilde p)=\tilde p$. We denote by $d$ the distance induced by the (lifted) Riemannian metric on $M\times\tilde N$.

\begin{lemma}
\label{lemma_homotopic_to_identity}
The diffeomorphism $\tilde f$ is homotopic to the model diffeomorphism $id_M\times\tilde L\colon (x,y)\mapsto (x, \tilde Ly)$ via a homotopy $H$. Moreover,
$d(H_t, id_M\times\tilde L)$ is uniformly bounded for all $t\in[0,1]$.
\end{lemma}
\begin{proof} 
By reduction R4 the diffeomorphism $f\colon M\times N\to M\times N$ and the model map $(x,y)\mapsto (x, Ly)$ 
induce the same outer automorphism of $\pi_1(M\times N)$ and hence are homotopic~\cite[Proposition 1B.9]{hatcher}. 
This homotopy lifts to the posited homotopy $H$ on $M\times \tilde N$.
The distance $d(H_t, id_M\times\tilde L)$ is bounded because $H$ is a lift of a homotopy on a compact manifold $M\times N$.
\end{proof}
We also lift the semi-conjugacy $h\colon M\times N\to N$ constructed in Subsection~\ref{sec_franks} to a 
semi-conjugacy $\tilde h\colon M\times \tilde N\to\tilde N$ and the automorphism $L\colon N\to N$ to an 
automorphism $\tilde L\colon\tilde N\to\tilde N$ so that $\tilde L(\tilde h(\tilde p))=\tilde h(\tilde p)$. Then we 
have the following commutative diagram.
$$
\xymatrix{
  *[r]{\;(M\times\tilde N,\tilde p)\;} \ar@<4ex>_\pi[d]\ar@<0.1ex>[rr]^-{\tilde h} \ar@(dl,ul)[]^{\tilde f}    
  & & *[l]{\;(\tilde N, \tilde h(\tilde p))\,}  \ar@(dr,ur)[]_{\tilde L} \ar@<-5.3ex>^{\bar\pi}[d]\\
    *[r]{\;(M\times N, p)\;} \ar@<0.1ex>[rr]^-{ h} \ar@(dl,ul)[]^{ f}    && *[l]{\;( N,  h( p))\,}  \ar@(dr,ur)[]_{ L} 
}
$$
We now define the set $\tilde K:=\tilde h^{-1}(\tilde h(\tilde p))$. Clearly $\tilde K$ is $\tilde f$-invariant.
\begin{lemma}
The restriction $\pi|_{\tilde K}\colon\tilde K\to K$ is a homeomorphism.
\end{lemma}
\begin{proof}
Indeed, to see that $\pi|_{\tilde K}$ is a bijection notice that
$$
\pi^{-1}(K)=\pi^{-1}\big(h^{-1}(h(p))\big)=\tilde h^{-1}\big(\bar\pi^{-1}(h(p))\big)
=\tilde h^{-1}\Big(\bigcup_{g\in G} g.\tilde h(\tilde p)\Big)\\
=\bigcup_{g\in G} g.\tilde K,
$$
where the union over the group of deck transformations is disjoint. Hence each $g.\tilde K$ projects homeomorphically 
onto $K$.
\end{proof}
This lemma together with the preceding commutative diagram provides a certain understanding of the dynamics of 
$\tilde f$ which we summarize below.

\begin{prop}
\label{prop_tilde_f}
Let $\tilde f\colon M\times\tilde N\to M\times\tilde N$ and $\tilde K$ be as above. We equip $M\times\tilde N$ 
with a Riemannian metric lifted from $M\times N$. Then
\begin{itemize}
\item $\tilde f$ is an Anosov diffeomorphism with infinitely many periodic points each of which belongs to the 
$\tilde f$-invariant compact set $\tilde K$;
\item $d_{C^0}(\tilde f, id_M\times \tilde L)<const$;
\item any point $x\notin\tilde K$ escapes to infinity either in positive or in negative time.
\end{itemize}
\end{prop}

\subsection{Fiberwise compactification} Our goal now is to change the diffeomorphism $\tilde f$ so that the set 
of periodic points stays the same, and the new map $\hat f$ (which is not necessarily a diffeomorphism) extends 
to the fiberwise compactification $M\times \S^k$, where $k=\dim N$. 

\renewcommand\labelenumi{P\theenumi.}
\begin{lemma}
Let $\tilde f\colon M\times\tilde N\to M\times\tilde N$ and $\tilde K$ be as above. Then there exists a map 
$\hat f\colon M\times\tilde N\to M\times\tilde N$ which has the following properties:
\begin{enumerate}
\item $\hat f$ coincides with $\tilde f$ on a large compact set $B_1$ which contains $\tilde K$;
\item $\hat f$ has infinitely many periodic points each of which belongs to $\tilde K$;
\item Outside of a large compact set $B_2\supset B_1$ 
$$
\hat f(x,y)=(x, \tilde L(y)).
$$
\end{enumerate}
\end{lemma}
\begin{proof}
We first choose a pair of large nested compact sets $D_i \subset \tilde N$, $i=1,2$, chosen so that each
$D_i$ is homeomorphic to $\DD^k$, and such that the intermediate region $D_2 \setminus Int(D_1)$ 
is homeomorphic to $\S^{k-1} \times [0,1]$. We fix an identification $D_2 \setminus Int(D_1)\rightarrow \S^{k-1} \times [0,1]$, giving us a 
coordinate system on the region $D_2\setminus Int(D_1)$. 

The sets $B_i$ will be given by $M\times D_i$.
Recall that $H_t: M\times \tilde N \rightarrow M\times \tilde N$
is the proper homotopy obtained by lifting the homotopy $f\simeq Id\times L$ to the cover $M\times \tilde N$
(parametrized so that $H_0 = Id_M\times \tilde L$ and $H_1 = \tilde f$). 
Now we define the map $\hat f$ via the following formula:
$$
\hat f(x,y)=
\begin{cases} \tilde f(x,y) & \mbox{if } (x,y)\in B_1 \\
H_{1-t}(x, y) & \mbox{if } (x,y)\in B_2  \setminus Int(B_1), y\in \S^{k-1}\times \{t\} \\
(x, \tilde L(y)) & \mbox{if } (x,y)\notin B_2
\end{cases}
$$
Note that on the slice $\S^{k-1}\times \{0\} = \partial B_1$, we have that $H_1 = \tilde f$, while on the slice
$\S^{k-1}\times \{1\} = \partial B_2$, we have $H_0 = Id_M\times \tilde L$. So the above map is indeed continuous.

Properties P1 and P3 are clear from the construction. To verify P2 note that 
\begin{equation}
\label{eq_hat_f}
d(\hat f, id_M\times\tilde L)<R
\end{equation}
where the constant $R$ is independent of the choices which we made while constructing $\hat f$, that is, $R$ is independent of the choice of $B_1$, $B_2$ and the 
coordinate system on $B_2\backslash B_1$. Indeed, this is clear from the estimate on the homotopy $H$ from Lemma~\ref{lemma_homotopic_to_identity}. 

Therefore, estimate~(\ref{eq_hat_f}) guarantees that for sufficiently large $B_1$ any point $(x,y)\notin B_1$ will escape to infinity in positive or negative time under 
$\hat f$, because it escapes to infinity under $id_M\times\tilde L$. It remains to notice that any point $(x,y)\in B_1\backslash\tilde K$ escapes $B_1$ in positive or 
negative time by the last statement in Proposition~\ref{prop_tilde_f} (because $\hat f|_{B_1}=\tilde f|_{B_1}$).
\end{proof}

Consider the space $M\times \S^k$ (where $k=\dim (N)$), obtained by taking the one point compactification of each
$\tilde N$-fiber in $M\times \tilde N$. 
The collection of compactifying points for each of the fibers give a copy $M_\infty$ of $M$, which we call 
the fiber at infinity. 
Since the map $\hat f$ restricted to each fiber $\{p\}\times \tilde N$ coincides with the map $\tilde L$ (outside of a compact set), we
can extend the map on each fiber to the one-point compactification. This gives us a new map, 
denoted $\hat f$, from $M\times \S^k$ to $M\times \S^k$. Note that the new map $\hat f$ restricts to 
the identity on $M_\infty$. We 
further homotope $\hat f$ in a small neighborhood of $M_\infty$ (keeping the $\tilde N$ coordinates constant)
so that the restriction $\hat f|_{M_\infty}$ becomes a time-one map of a gradient flow. Then 
the following Proposition is immediate from P2, the discussion in~Section~\ref{Sec_Lef} and the Poincar\'e-Hopf formula~(\ref{eq_euler}).
\begin{prop}
\label{prop_lefschetz_number}
Let $\hat f\colon M\times\S^k\to M\times\S^k$ be the map constructed above. Then the Lefschetz numbers can 
be calculated as follows
$$
\Lambda(\hat f^m)=\sum_{p\in Fix(\hat f^m)}ind_{\hat f^m}(p)=(-1)^{s}\chi(M)+(-1)^{\dim E^u}\left|Fix(\hat f^m|_{\tilde K})\right|,
$$
where $s$ is the dimension of the stable subspace of $L$. In particular, as $m$ tends to infinity, $\Lambda(\hat f^m)$ is unbounded. (cf. Section~\ref{sec_34}.)
\end{prop}

\subsection{Lefschetz numbers on $M\times \S^k$}

We have now built a self-map $\hat f: M\times \S^k \rightarrow M\times \S^k$. {As explained in the last section, }
from the construction of $\hat f$, we have
a good understanding of the periodic points. On the other hand, we can easily compute the Lefschetz numbers for the map
$\hat f$ and its powers.

\begin{lemma}\label{lem:action-on-cohomology}
The map $\hat f$ is homotopic to the identity map on $M\times \S^k$. 
\end{lemma}

\begin{proof}
Reversing the small homotopy near $M_\infty$, we can assume that $\hat f$ coincides with $Id_M\times \tilde L$
on the complement of $B_2$ (and is the identity on $M_\infty$). Next we define a homotopy $\bar H$ on the
space $M\times \tilde N \subset M\times \S^k$ as follows:
$$
\bar H_s(x,y)=
\begin{cases} H_s(x,y), \;\;\mbox{if}\;\;\; (x,y)\in B_1 \\
H_{s(1-t)}(x, y), \;\;\mbox{if}\;\;\; (x,y)\in B_2  \setminus Int(B_1), y\in \S^{k-1}\times \{t\} \\
H_0(x, y), \;\;\mbox{if}\;\;\; (x,y)\notin B_2
\end{cases}
$$
Let us observe that $\bar H_0 = H_0 = Id_M\times \tilde L$, while $\bar H_1 = \hat f$. Moreover this 
is a proper homotopy, whose support is entirely contained in the compact set $B_2$. So it extends to the fiberwise
one-point compactification, giving us a homotopy from the self-map $\hat f$ to the map 
consisting of the fiberwise one-point compactification of $Id_M\times \tilde L$. Finally, we note that $\tilde L$ is properly
homotopic to $Id_{\tilde N}$. Composing with these (fiber-wise) homotopies, we obtain the desired homotopy, completing
the proof of the Lemma.
\end{proof}

As an immediate consequence, we obtain the following.

\begin{cor}\label{cor:Lefschtz-constant}
For all $m\geq 0$, we have that $\Lambda (\hat f ^m) = \Lambda(\text{Id}|_{M\times \S^k}) = \chi(M)\chi(\S^k)$ is a constant,
independent of $m$.
\end{cor}

But this contradicts the computation of the Lefschetz number obtained in Proposition \ref{prop_lefschetz_number},
completing the proof of the {\bf Main Theorem}.

\section{Applications and concluding remarks}\label{sec:conclusion}

\subsection{Concrete examples}\label{sec_examples}

We now prove Corollary \ref{main-corollary}. We first consider the case of a single factor. 
Let $M^n$ ($n\geq 3$) be either (a) an aspherical manifold with Gromov hyperbolic 
fundamental group, or (b) a higher rank locally symmetric space with no local $\mathbb R$ or $\mathbb H^2$-factors. 
We need to verify that the fundamental group $\Gamma$ of $M$ satisfies conditions (i)-(iii) in the statement of our 
{\bf Main Theorem}.

Property (i) states that $\Gamma$ is Hopfian. For torsion-free Gromov hyperbolic groups, this is a deep result of
Sela \cite{Sela}. For higher rank locally symmetric spaces, the fundamental group is linear (i.e. has a 
faithful finite-dimensional complex representation). But Mal'cev \cite{Mal2} 
showed that finitely generated linear groups are Hopfian.

Property (ii) states that $\text{Out}(\Gamma)$ is finite. For lattices in higher-rank semisimple Lie groups, this is a
direct consequence of Mostow's rigidity theorem: $\text{Out}(\Gamma)$ coincides with the compact Lie group
$\text{Isom}(M)$. But the latter must be $0$-dimensional (hence finite), for otherwise the action of a $1$-parameter subgroup 
would give a local $\mathbb R$-factor in $M$. In the case where $\Gamma$ is Gromov hyperbolic, the work of 
Paulin \cite{Pa} and Bestvina and Feighn \cite{BF} reduce the problem to deciding whether the boundary at infinity 
$\partial ^\infty \Gamma$ has any local cutpoints (see the discussion in the proof of Corollary \ref{cor-yano}).
But Bestvina has shown that $\partial ^\infty \Gamma$ is a homology manifold of dimension $n-1\geq 2$ 
\cite[Theorem 2.8]{Be}.
Thus we can find arbitrarily small open neighborhoods $U$ around any given point $p\in \partial ^\infty \Gamma$, 
with the property that $H_1(U, U\setminus \{p\})\cong H_0(U, U \setminus \{p\}) \cong 0$. The homology long exact 
sequence for the pair $(U, U\setminus \{p\})$ then shows that $U$ and $U\setminus \{p\}$ have the same number of
connected components, and hence $p$ cannot be a local cutpoint.

Finally, property (iii) concerns the intersection of maximal nilpotent subgroups of $\Gamma$. When $\Gamma$ is
torsion-free Gromov hyperbolic, these are precisely the (maximal) cyclic subgroups of $\Gamma$. 
Since $M$ has dimension
$>1$, $\Gamma$ is not virtually cyclic, so must contain a free subgroup. In particular, we have a pair of cyclic 
subgroups that intersect trivially, verifying (iii). In the case where $M$ is an irreducible higher rank locally 
symmetric space, the intersection of all the maximal nilpotent subgroups forms 
a normal subgroup of $\Gamma=\pi_1(M)$. From Margulis' normal subgroup theorem 
(see \cite[Ch. IV]{margulis}, or \cite[Ch. 8]{zimmer})
this intersection is either finite, or 
finite index. Since the intersection is nilpotent, it cannot be finite index in $\Gamma$, hence it must be finite. Since
$\Gamma$ is torsion-free, it follows that it is trivial.

\vskip 5pt

Now we pass to the case where $M= M_1\times \cdots \times M_k$ is a product of manifolds of the above type.
We again need to verify that the fundamental group $\Gamma = \Gamma_1 \times \cdots \times \Gamma_k$ of
such a manifold verifies the algebraic conditions (i)-(iii) from our {\bf Main Theorem.} 
For (iii) we observe that a product of nilpotent groups is nilpotent, and the homomorphic image of 
a nilpotent group is nilpotent. It follows that the maximal nilpotent subgroups
in $\Gamma$ are precisely the products of the maximal nilpotent subgroups in the individual $\Gamma_i$. From
this, one can see that the intersection of the maximal nilpotent subgroups in $\Gamma$ coincides with the product
of the intersection of the maximal nilpotents in the individual $\Gamma_i$. But we saw earlier that the later are all
trivial, which shows property (iii).

Properties (i) and (ii) are harder to establish. As a first step, let us analyze 
surjective homomorphisms $\phi: \Gamma \rightarrow \Gamma$. By an abuse of notation, 
we will also denote by $\Gamma_i$ the subgroup of $\Gamma$ consisting of all tuples $(g_1, \ldots, g_k)$ 
with $g_j=id$ when $j\neq i$. Restricting the homomorphism $\phi$ to the component $\Gamma_i$, we can 
express $\phi|_{\Gamma_i} = (\phi_{i1}, \ldots ,\phi_{ik})$, where each $\phi_{ij}: \Gamma_i \rightarrow \Gamma_j$ 
is a homomorphism (not necessarily surjective). The homomorphism $\phi$ is then completely determined by the
collection of homomorphisms $\phi_{ij}$ (where $1\leq i, j \leq k$), since we have the formula:
$$\phi\big( (g_1, \ldots , g_k)\big) = \Big( \prod_{i=1}^k\phi_{i1}(g_i), \prod_{i=1}^k\phi_{i2}(g_i), \cdots , 
\prod_{i=1}^k\phi_{ik}(g_i)\Big).$$
Note that, in the expression above, the products are unambiguous: since the subgroups $\Gamma_i$ inside
$\Gamma$ all commute, their $\phi$-images also commute, so the order in which the product is taken does not
matter. We now define the image subgroups $\Lambda_{ij}:= \phi_{ij}(\Gamma_i) \leq \Gamma_j$. 
Since $\Gamma_i$ is normal inside 
$\Gamma$, and the morphism $\phi$ is surjective, each of the image subgroups $\Lambda_{ij}$ are normal inside 
the corresponding $\Gamma_j$. 

\vskip 10pt

We will now try to understand the individual $\Lambda_{st}$. If $\Gamma_t$ is an irreducible lattice in a higher 
rank semisimple Lie group, then from Margulis' normal subgroup theorem, it follows that each $\Lambda_{rt}$ is either
finite (whence trivial) or finite index in $\Gamma_t$. We claim exactly one of these is finite index (and must hence have 
index one). Indeed, if we have two distinct $\Lambda_{rt}$ and $\Lambda_{st}$ 
($r\neq s$) both of which are finite index, then their intersection $\Lambda_{rt}\cap \Lambda_{st}$ has finite
index in $\Gamma_t$. Since $\Gamma_r, \Gamma_s$ commute in $\Gamma$, their images commute inside
$\Gamma_t$, and thus this intersection must be an abelian subgroup. So one obtains that $\Gamma_t$ must be
virtually abelian, a contradiction.

Next let us assume that $\Gamma_t$ is torsion-free Gromov hyperbolic. Again, since the image subgroups $\Lambda_{rt}$ 
and $\Lambda_{st}$ commute inside $\Gamma_t$, we have that $\Lambda_{rt}$ is actually a subgroup of the centralizer
$\text{Cent}_{\Gamma_t}(\Lambda_{st})$. We will need the following:

\vskip 5pt

\begin{lemma}\label{centralizer} Let $H$ be a subgroup in a torsion-free Gromov hyperbolic group $G$,
\renewcommand\labelenumi{(\theenumi)}
\begin{enumerate}
\item if the subgroup $H$ is trivial, the centralizer is the whole group $G$.
\item if the subgroup $H$ is (infinite) cyclic, then every element in the centralizer generates a cyclic subgroup
which intersects $H$ non-trivially.
\item if $H$ is not cyclic, then the centralizer is trivial. 
\end{enumerate}
\end{lemma}

\begin{proof}
Statement (1) is obvious. If $H$ contains an element $x$ of infinite order, then any element $y\in \text{Cent}_G(H)$ in the 
centralizer must 
commute with $x$. Since $G$ cannot contain any $\mathbb Z^2$-subgroups, this implies that the torsion-free abelian subgroup 
$\langle x,y \rangle$ generated by the two elements is isomorphic to $\mathbb Z$. If $H$ is infinite cyclic, this immediately
gives (2). 

Next, if $H$ is {\it not} infinite cyclic, then it contains two elements $a,b$ which do not jointly
lie inside a $\mathbb Z$ subgroup (recall that a finitely generated locally cyclic group is cyclic). From the Tits' alternative,
there exists a $k>0$ with the property that $\langle a^k, b^k\rangle$ is a free group (see \cite[Prop. 3.20, pg. 467]{bridson-haefliger}). 
Now if $z$ is an arbitrary non-trivial element in the centralizer of $H$, then from the discussion above, 
$\langle z\rangle \cap \langle a^k\rangle$ and $\langle z\rangle \cap \langle b^k \rangle$ both have finite index 
inside $\langle z \rangle$, so $\langle a^k \rangle \cap \langle b^k\rangle \cong \mathbb Z$ inside the free
group $\langle a^k, b^k\rangle$, a contradiction. 
This establishes statement (3).
\end{proof}

We now finish our analysis of the possible $\Lambda_{rt}$ when $\Gamma_t$ is Gromov hyperbolic. We first
note that none of the $\Lambda_{rt}$ can be infinite cyclic. Indeed, if $\Lambda_{rt}$ is infinite cyclic, then from 
the discussion above, we obtain that all the $\Lambda_{st}$ ($s\neq r$) lie inside the centralizer of
$\Lambda_{rt}$. Hence, the subgroup spanned by them also lies inside the centralizer. But by 
surjectivity of the map $\phi$, this means that the entire (finitely generated) group $\Gamma_t$ coincides
with the centralizer.  Property (2) in Lemma \ref{centralizer} would then imply that $\Gamma_t$ is a torsion-free $2$-ended group (i.e. 
$\partial ^\infty \Gamma_t$ consists of just two points), so must be isomorphic
to $\mathbb Z$. This is impossible, as $\Gamma_t$
is the fundamental group of an aspherical manifold of dimension $>1$. 

Having showed that none of the $\Lambda_{rt}$ are {\it infinite} cyclic, we now note that {\it at least one} of
the $\Lambda_{rt}$ must be non-cyclic. For otherwise all the $\Lambda_{rt}$would be finite cyclic, and hence trivial,
contradicting the surjectivity of $\phi$. So assume $\Lambda_{rt}$ is not cyclic. Then from statement (3) in Lemma \ref{centralizer}, 
one sees that all the remaining $\Lambda_{st}$ ($s\neq t$) must be trivial. Finally, surjectivity of $\phi$ then forces that 
$\Lambda_{rt} = \Gamma_t$. We summarize the discussion so far in the following:

\vskip 5pt

\noindent {\bf Fact:} For each $j$, there is a unique $1\leq \tau(j) \leq k$ with the property that
$$\Lambda_{ij} = \begin{cases}
\Gamma_j & i=\tau(j) \\
1 & i\neq \tau(j) \\
\end{cases}$$

\vskip 5pt

\noindent Since we only have finitely many factors in $\Gamma$, surjectivity of $\phi$ implies that $\tau$ is a bijection, 
i.e. defines an element in the symmetric group $\text{Sym}(k)$. We denote by $\sigma := \tau ^{-1}$ the inverse
permutation. 

Let us first show the Hopfian property (i). From the discussion above, to any surjective homomorphism $\phi: \Gamma
\rightarrow \Gamma$, we have associated a permutation $\sigma\in \text{Sym}(k)$, which encodes how the surjective 
homomorphism permutes the direct factors in $\Gamma$. It follows that a high enough power $\phi^r:\Gamma\rightarrow
\Gamma$ is a surjective homomorphism which preserves each factor, and hence induces a surjection from each
factor to itself. As we know each factor is Hopfian, it follows that $\phi^r$ is injective, and hence $\phi$ must also be 
injective. Thus $\phi$ is an isomorphism, and we conclude that $\Gamma$ is Hopfian, verifying (i).

For (ii), we note that the above discussion gives us a well-defined homomorphism $\text{Aut}(\Gamma)\rightarrow 
\text{Sym}(k)$. Since an inner automorphism preserves each factor, this descends to a homomorphism 
$\text{Out}(\Gamma) \rightarrow \text{Sym}(k)$. The image of this homomorphism lies in a finite group, while the
kernel can easily be seen to be $\text{Out}(\Gamma_1) \times \cdots \times \text{Out}(\Gamma_k)$. Since each
of the $\text{Out}(\Gamma_i)$ are finite, we obtain that $\text{Out}(\Gamma)$ is also finite, establishing (ii). This 
concludes the proof of Corollary \ref{main-corollary}.

\begin{remark}
The obstruction mentioned in Lemma \ref{lem:finite-out} does not apply to the examples in our 
Corollary \ref{main-corollary}. Indeed, the outer automorphism group of $\pi_1(M\times N)$ contains the outer
automorphism group of $\pi_1(N)$. If $N$ is an infranilmanifold supporting an Anosov diffeomorphism, then the
later has infinite order, and hence so does $\text{Out}(\pi_1(M\times N))$.
\end{remark}

\subsection{Open problems}

There are a number of interesting questions we encountered on this project.
Concerning Smale's question, it seems like the following special case is still open:

\begin{problem}
Let $M$ be a compact non-positively curved locally symmetric manifold and let $f\colon M\to M$ be an Anosov diffeomorphism. Then a finite cover of $M$ is a torus.
\end{problem}

The crucial difficulty in resolving this problem seems to be the following:

\begin{problem}
Show that the product of two surfaces at least one of which has genus $\ge 2$ does not support an Anosov diffeomorphism.
\end{problem}

Finally, regarding Lemma \ref{lem:finite-out}, it seems quite hard to produce high-dimensional aspherical manifolds 
whose fundamental group has infinite outer automorphism group. While several classes of aspherical manifolds 
are known to have fundamental groups with infinite outer automorphism group (see e.g. \cite{DM}, \cite{MS}, \cite{NP}, 
\cite[Sections 5.2, 5.4]{FLS}), they all appear to contain $\mathbb Z^2$-subgroups.
In particular, we do not know the answer to the following:

\begin{question}
Let $M$ be an aspherical manifold of dimension $\geq 3$, and assume $\Gamma:= \pi_1(M)$ contains no 
$\mathbb Z^2$-subgroup. Does it follow that $\text{Out}(\Gamma)$ is finite?
\end{question}

\begin{remark}
This question has a positive answer for 3-dimensional manifolds. 
{
Indeed, one notes that such a manifold $M$ is aspherical and atoroidal. Hence $M$ is prime, and has trivial JSJ 
decomposition. From geometrization, it follows that $M$ must be modeled on one of the eight $3$-dimensional 
geometries: $\mathbb H^3$, $\mathbb E^3$, $\mathbb S^3$, $\mathbb S^2\times \mathbb E^1$, 
$\mathbb H^2\times \mathbb E^1$, $\widetilde{PSL}(2,\mathbb R)$, $NIL$, or $SOL$. Since $M$ is aspherical,
we can rule out the two model geometries $\mathbb S^3$ and $\mathbb S^2\times \mathbb E^1$. 
And since $\pi_1(M)$ has no $\mathbb Z^2$-subgroup, we can rule out the five model geometries $\mathbb E^3$, 
$\mathbb H^2\times \mathbb E^1$, $\widetilde{PSL}(2,\mathbb R)$, $NIL$, and $SOL$. Thus $M$ supports a 
hyperbolic metric, and finiteness of $\text{Out}(\Gamma)$ follows from Mostow rigidity. }

\end{remark}

\begin{remark}
{
A famous question of Gromov asks whether CAT(0) groups with no $\mathbb Z^2$-subgroup are Gromov hyperbolic. 
An affirmative answer to Gromov's question would then imply, by a standard argument (see Corollary \ref{cor-yano}
and the discussion at the beginning of Section \ref{sec_examples}) an affirmative answer to our question as well.  
}
\end{remark}


\begin{thebibliography}{texttLL}

\bibitem[Be96]{Be} M. Bestvina, {\it Local homology properties of boundaries of groups.} Michigan Math. J. {\bf 43} (1996),
no. 1, 123--139.

\bibitem[BF95]{BF} M. Bestvina, M. Feighn, {\it Stable actions of groups on real trees.} Invent. Math. {\bf 121} (1995), no. 2,
287--321.

\bibitem[Bo98]{Bow} B. Bowditch, {\it Cut points and canonical splittings of hyperbolic groups.} Acta Math. {\bf 180} (1998),
no. 2, 145--186.

\bibitem[BH99]{bridson-haefliger} M. R. Bridson, A. Haefliger, {\it Metric spaces of non-positive curvature,} Grundlehren der 
mathematischen Wissenschaften {\bf 319}, Springer-Verlag, 1999.

\bibitem[Br77]{Br} M. Brin, {\it Nonwandering points of Anosov diffeomorphisms.} 
Dynamical systems (Vol. 1), 11--18, Ast\'erisque {\bf 49}, Soc. Math. France, Paris, 1977. 

\bibitem[BM81]{BrM} M. Brin, A. Manning, {\it Anosov diffeomorphisms with pinched spectrum.} 
Dynamical systems and turbulence (Warwick 1980), 48--53, Lecture Notes in Math. {\bf 898}, 
Springer, Berlin-New York, 1981.

\bibitem[BS02]{BS} M. Brin, G. Stuck, {\it Introduction to dynamical systems.} Cambridge University Press, 
Cambridge, 2002. xii+240 pp.

\bibitem[DM95]{DM} K. Dekimpe, W. Malfait, {\it Almost-crystallographic groups with many outer automorphisms}. 
Comm. Algebra {\bf 23} (1995), no. 8, 3073--3083.

\bibitem[D16]{D} J. Der\'e, {\it A new method for constructing Anosov Lie algebras. } Trans. Amer. Math. Soc. 368 (2016), no. 2, 1497--1516. 


\bibitem[FJ90]{FJ} F.T. Farrell, L.E. Jones, {\it Smooth nonrepresentability of $Out(\pi_1M)$}. 
Bull. London Math. Soc. {\bf 22} (1990), no. 5, 485--488.

\bibitem[Fr69]{Fr2} J. Franks, {\it
Anosov diffeomorphisms on tori.}
Trans. Amer. Math. Soc. 145 1969 117--124. 

\bibitem[Fr70]{Fr} J. Franks, {\it Anosov diffeomorphisms.} 
Global Analysis, Proc. Sympos. Pure Math {\bf 14}, AMS, Providence, R.I. 1970, 61--93.



\bibitem[FLS15]{FLS} R. Frigerio, J.-F. Lafont, A. Sisto, {\it Rigidity of high dimensional graph manifolds}. 
Ast\'erisque {\bf 372}, Soc. Math. France, Paris, 2015. xxii+171 pp.

\bibitem[GRH14]{GRH} A. Gogolev, F. Rodriguez Hertz, 
{\it Manifolds with higher homotopy which do not support Anosov diffeomorphisms.} 
Bull. Lond. Math. Soc. {\bf 46} (2014), no. 2, 349--366.

\bibitem[G83]{Goo} S. Goodman, {\it Dehn surgery on Anosov flows.}
Geometric dynamics (Rio de Janeiro, 1981), 300--307, Lecture Notes in Math. {\bf 1007}, Springer, Berlin, 1983.

\bibitem[Ha02]{hatcher} A. Hatcher, {\it Algebraic Topology.} Cambridge University Press, 2002.

\bibitem[Hi71]{H} M. W. Hirsch, {\it Anosov maps, polycyclic groups, and homology.} 
Topology {\bf 10} (1971), 177--183.

\bibitem[KH95]{KH} A. Katok, B. Hasselblatt, {\it Introduction to the modern theory of dynamical systems.} 
Cambridge University Press, 1995.

\bibitem[LW09]{LW} J. Lauret, C.E. Will, {\it Nilmanifolds of dimension $\le8$ admitting Anosov diffeomorphisms.} Trans. Amer. Math. Soc. 361 (2009), no. 5, 2377--2395.

\bibitem[Mal40]{Mal2} A. I. Mal$'$cev, {\it On the faithful representation of infinite groups by matrices}. 
Mat. Sb., {\bf 8 (50)} (1940), 405--422.

\bibitem[Mal49]{Mal} A. I. Mal$'$cev, {\it On a class of homogeneous spaces.}  
Izvestiya Akad. Nauk. SSSR. Ser. Mat. {\bf 13} (1949), 9--32.

\bibitem[MS96]{MS} W. Malfait, A. Szczepa\'nski, {\it Almost-Bieberbach groups with (in)finite outer automorphism 
group}. Glasgow Math. J. {\bf 40} (1998), no. 1, 47--62.

\bibitem[Man73]{Mann1} A. Manning, {\it Anosov diffeomorphisms on nilmanifolds.} 
Proc. Amer. Math. Soc. {\bf 38} (1973), 423--426.

\bibitem[Man74]{Mann2} A. Manning, {\it There are no new Anosov diffeomorphisms on tori.} 
Amer. J. Math. {\bf 96} (1974), 422--429.

\bibitem[Mar91]{margulis} G. A. Margulis, {\it Discrete subgroups of semisimple Lie groups,}  
Ergebnisse der Mathematik und ihrer Grenzgebiete {\textbf 17}, Springer, 1991.



\bibitem[N70]{N} S. E. Newhouse, {\it On codimension one Anosov diffeomorphisms. }
Amer. J. Math. {\bf 92} (1970), 761--770. 

\bibitem[NP12]{NP} T. T. Nguyen-Phan, {\it Smooth (non)rigidity of cusp-decomposable manifolds}. 
Comment. Math. Helv. {\bf 87} (2012), no. 4, 789--804.

\bibitem[P91]{Pa} F. Paulin, {\it Outer automorphisms of hyperbolic groups and small actions on ${\mathbb R}$-trees.} 
Arboreal Group Theory, pp. 331--343. MSRI Publ. {\bf 19}, Springer-Verlag, 1991.


\bibitem[RS75]{RS} D. Ruelle, D. Sullivan, {\it Currents, flows and diffeomorphisms. } 
Topology {\bf 14} (1975), no. 4, 319--327. 


\bibitem[Se99]{Sela} Z. Sela, {\it Endomorphisms of hyperbolic groups. I. The Hopf property.} 
Topology {\bf 38} (1999), no. 2, 301--321.

\bibitem[Sh73]{Sh} K. Shiraiwa, {\it Manifolds which do not admit Anosov diffeomorphisms. }
Nagoya Math. J. {\bf 49} (1973), 111--115. 

\bibitem[Sm67]{Sm} S. Smale, {\it Differentiable dynamical systems.} 
Bull. Amer. Math. Soc. {\bf 73} (1967), 747--817.
\bibitem[V73]{vick} J. W. Vick, {\it Homology theory. An introduction to algebraic topology.} 
Pure and Applied Mathematics {\bf 53}. Academic Press, New York-London, 1973. xi+237 pp. 

\bibitem[Y83]{Y} K. Yano, {\it There are no transitive Anosov diffeomorphisms on negatively curved manifolds. }
Proc. Japan Acad. Ser. A Math. Sci. {\bf 59} (1983), no. 9, p. 445. 

\bibitem[Z84]{zimmer} R. J. Zimmer, {\it Ergodic Theory and Semisimple Groups}, 
Monographs in Mathematics {\textbf 81}, Birkha\"user, 1984.

\end{thebibliography}
\end{document}